\documentclass[11pt,a4paper]{article}
 \usepackage{euscript}
\usepackage{amsmath}
\usepackage{graphicx}
\usepackage{amsthm}
\usepackage{amssymb}
\usepackage{epsfig}
\usepackage{url}
\usepackage{colonequals}
\usepackage{amsmath,amscd}
\usepackage{yfonts}
\usepackage{txfonts}
\usepackage{mathrsfs}
\theoremstyle{theorem}
\newtheorem{theorem}{Theorem}[section]

\newtheorem{proposition}[theorem]{Proposition}

\newtheorem{definition}{Definition}[section]

\newtheorem{remark}{Remark}[section]
\numberwithin{equation}{section}

\title{A new notion of semiprime submodules} 
\author{Masood Aryapoor\\
\tiny{\textit{Division of Mathematics and Physics}}\\
\tiny{\textit{M\"{a}lardalen  University}}\\
\tiny{\textit{Hamngatan 15, 632 17, Eskilstuna, 
Sweden
}}
}
 \date{}
\begin{document}

 \maketitle
\begin{abstract}
\noindent
	We introduce a new concept of a semiprime submodule. We show that a submodule of a finitely generated module over a commutative ring is semiprime if and only if it is radical, that is, an intersection of prime submodules. Using our notion, we also provide a new characterization of radical submodules of finitely generated modules over commutative rings. 
\end{abstract}
%%%%%%%%%%%%%%%%%%%%%%%%%%%%%%%%%%%%%%%%%%%%%%%%%%%%%%%%%%%%%%%%
%%%%%%%%%%%%%%%%%%%%%%%%%%%%%%%%%%%%%%%%%%%%%%%%%%%%%%%%%%%%%%%%
%%%%%%%%%%%%%%%%%%%%%%%%%%%%%%%%%%%%%%%%%%%%%%%%%%%%%%%%%%%%%%%%
%%%%%%%%%%%%%%%%%%%%%%%%%%%%%%%%%%%%%%%%%%%%%%%%%%%%%%%%%%%%%%%%
\begin{section}{Introduction} 

A submodule $N$ of a  module $M$ over a commutative ring $R$ is called prime if $rm\in N$ implies $rM\subset N$ or $m\in N$  for all $r\in R$ and $m\in N$. The notion of a prime submodule effectively generalizes  the concept of a prime ideal. This notion has been extensively examined and studied by various authors, leading to significant results. See \cite{dauns1978prime} for early work on the subject in the noncommutative case.

An ideal $I$ of  $R$ is called semiprime if $r^2\in I$ implies $r\in I$ for all $r\in R$. It is a fundamental property of semiprime ideals that a proper ideal $I$ of $R$ is semiprime if and only if it is radical, that is, it is an intersection of prime ideals of $R$. Generalizing the concept of semiprime ideals to modules in an obvious way leads us to the following definition: A submodule $N$ of a left $R$-module $M$ is called semiprime if $r^2m\in N$ implies $rm\subset N$ for all $r\in R$ and $m\in N$ (see \cite[Definition 1.9]{dauns1978prime}). It is easy to see that this condition is equivalent to the condition that the ideal $(N:m)\colonequals \{r\in R\,|\, rm\subset N \}$ is a semiprime ideal for all $m\in M$ (see also \cite[Lemma 3.1]{saracc2009semiprime}). This notion of a semiprime submodule was introduced by Dauns in the noncommutative setting \cite{dauns1978prime}. Although Dauns' definition of a semiprime module lacks the aforementioned property of semiprime ideals,  conditions under which a semiprime submodule is an intersection of prime submodules are known (see, for example, \cite{man1998commutative,saracc2009semiprime}).   

In a recent article \cite{CimpriPrime}, J. Cimprič introduced a new notion of semiprimeness for submodules of finitely generated free modules over a commutative ring, demonstrating that his notion satisfies the fundamental property of semiprime ideals. Cimprič noted that "We have not found yet a satisfactory extension of this definition to
general R-modules". Our main goal is to introduce an extension of his definition (see Definition \ref{def: semiprime}). Our main result is that our  definition of semiprimeness retains the fundamental property of semiprime ideals in the case of finitely generated modules over a commutative ring (see Theorem \ref{thm:radical = semiprime}). We also present a description of the prime radical of a submodule of a finitely generated module over a commutative ring (see Theorem \ref{thm:description of semiprime}).  

Throughout this paper, all rings are unital, associative, and commutative. Furthermore, all modules are assumed to be unital. 
\end{section} 
%%%%%%%%%%%%%%%%%%%%%%%%%%%%%%%%%%%%%%%%%%%%%%%%%%%%%%%%%%%%%%%%
%%%%%%%%%%%%%%%%%%%%%%%%%%%%%%%%%%%%%%%%%%%%%%%%%%%%%%%%%%%%%%%%

%%%%%%%%%%%%%%%%%%%%%%%%%%%%%%%%%%%%%%%%%%%%%%%%%%%%%%%%%%%%%%%%
%%%%%%%%%%%%%%%%%%%%%%%%%%%%%%%%%%%%%%%%%%%%%%%%%%%%%%%%%%%%%%%%
%%%%%%%%%%%%%%%%%%%%%%%%%%%%%%%%%%%%%%%%%%%%%%%%%%%%%%%%%%%%%%%%
%%%%%%%%%%%%%%%%%%%%%%%%%%%%%%%%%%%%%%%%%%%%%%%%%%%%%%%%%%%%%%%%

\begin{section}{Semiprime submodules of a module}
	
We begin by introducing our new notion of a semiprime submodule. 
\begin{definition}\label{def: semiprime}
	Let $R$ be a commutative ring and $M$ be an $R$-module. 
	A submodule $N$ of $M$ is called \textbf{semiprime} if $m\in (N:m)M$ implies $m\in N$ for all $m\in M$. 
\end{definition}
Here, the notation $(N:m)$ stands for the ideal $\{r\in R\, |\, rm\in N\}$. In what follows, we fix a commutative ring $R$ and  drop the prefix $R$ in the term ``$R$-module". 
  	
%%%%%%%%%%%%%%%%%%%%%%%%%%%%%%%%%%%%%%%%%%%%%%%%%%%%%%%%%%%%%%%%%   
%%%%%%%%%%%%%%%%%%%%%%%%%%%%%%%%%%%%%%%%%%%%%%%%%%%%%%%%%%%%%%%%% 

\begin{subsection}{Elementary properties of semiprime submodules}
	
	  We first show that Definition \ref{def: semiprime} generalizes the definition of semiprimeness introduced in  \cite{CimpriPrime}. 
	\begin{proposition}
		A submodule $N$ of the free module $R^n$ is semiprime in the sense of  Definition \ref{def: semiprime} if and only if for  any $m=(r_1,\dots,r_n)\in R^n$ such that $r_1m,\dots,r_nm\in N$, we have $m\in N$. In particular, and ideal of $R$ is semiprime as a submodule of  $R$ if and only if it is semiprime as an ideal of $R$. 
	\end{proposition}
	\begin{proof}
		Assume that $N$ is a semiprime submodule of  $R^n$ according to  Definition \ref{def: semiprime}. Let $m=(r_1,\dots,r_n)\in R^n$ satisfy the condition  $r_1m,\dots,r_nm\in N$, i.e., $r_1,\dots,r_n\in (N:m)$. It follows that $m\in (N:m)R^n$, which implies $m\in N$. This completes the proof of the forward direction.
		
		Conversely, assume that $N$ satisfies the property given in the proposition. Let $m=(r_1,\dots,r_n)\in R^n$ satisfy $m\in (N:m)R^n$. It follows that $m=\sum_i s_im_i$ where $s_i\in (N:m)$ and $m_i\in R^n$. Since $(N:m)$ is an ideal, the identity  $m=\sum_i s_im_i$ implies $r_i\in (N:m)$ for all $i$. Therefor, $r_1m,\dots,r_nm\in N$, from which it follows that $m\in N$. This complete the proof of the reverse direction. 
	\end{proof}
	Next, we prove some elementary properties of semiprime modules, generalizing some of the basic properties of semiprime ideals. 
	\begin{proposition}\label{prop:intersection of semiprime}
		The intersection of any family of semiprime submodules of $M$ is a semiprime submodule of $M$.
	\end{proposition}
	\begin{proof}
		The result follows easily from the observation that $(N:m)\subset (N':m)$ for all $m\in M$ and submodules $N\subset N'$ of $M$. 
	\end{proof}
	 The following result shows that every prime submodule of a module $M$ is a semiprime submodule of $M$. 
	\begin{proposition}\label{prop:prime implies semiprime}
		Let $P$ be a prime submodule of a module $M$. Then $P$ is a semiprime submodule of $M$.
	\end{proposition}
	\begin{proof}
		Assume that $m\in (P:m)M$ for some $m\in P$. We need to show that $m\in P$. If $(P:m)\subset (P:M)$, then $m\in (P:M)M\subset P $, and we are done. Suppose that $(P:m)\neq  (P:M)$. It follows that there exist $r\notin (P:M)$ such that $rm\in P$. Since $P$ is prime and $r\notin (P:M)$, we conclude that $m\in P$, completing the proof. 
	\end{proof}
		The following proposition shows that every semiprime submodule is semiprime in the sense of Definition 1.9 in \cite{dauns1978prime}. 
	\begin{proposition}\label{prop:(N:M) is radical}
		Let $N$ be a semiprime submodule of $M$. For all $r\in R$ and $m\in M$, if $r^2m\in N$, then $rm\in N$. In  other words, $(N:m)$ is a semiprime ideal for all $m\in  M$. 
	\end{proposition} 
	\begin{proof}
		If $r^2m\in N$, then $r\in (N:rm)$, implying 
		$rm\in (N:rm)M$. Since $N$ is semiprime, we conclude that $rm\in N$. 
	\end{proof}
	The next result characterizes semiprime submodules of quotient modules. 
	\begin{proposition}\label{prop:correspondence}
		Let  $M'$ be a submodule of $M$. Then the assignment $N\mapsto N/M'$ establishes a 1-1 correspondence between the set of semiprime submodules of $M$ containing $M'$ and the set of semiprime submodules of $M/M'$. 
	\end{proposition}
	\begin{proof}
		Let $M'\subset N$ be a semiprime submodule of $M$. We need to show that $N/M'$ is a semiprime submodule of $M/M'$. Let $m+M'\subset (N/M':m+M')M/M'$. Then 
		$m\in (N:m)M+M'$ because $(N/M':m+M')=(N:m)$. It follows that $m+m'\in (N:m)M$ for some $m'\in M'$. Since $(N:m+m')=(N:m)$ and $N$ is semiprime, $m+m'$ must belong to $N$, from which it follows that $m+M'\in N/M'$.  This proves that $M'/N$ is semiprime.It is left to the reader to complete the proof.
	\end{proof}
	We remark that a similar result holds for prime submodules, the proof of which is straightforward. See also \cite [Result 1.2]{McCaslandMoore}.
	
\end{subsection}
%%%%%%%%%%%%%%%%%%%%%%%%%%%%%%%%%%%%%%%%%%%%%%%%%%%%%%%%%%%%%%%%%   
%%%%%%%%%%%%%%%%%%%%%%%%%%%%%%%%%%%%%%%%%%%%%%%%%%%%%%%%%%%%%%%%% 

\begin{subsection}{Radical submodules}
	The (prime) radical of a submodule $N$ of $M$ is defined to be the intersection of all prime submodules of $M$ that contain $N$. If there does not exist a prime submodule of $M$ that contains $N$, the radical of $N$ is defined to be $M$. 	The radical of $N$ is denoted by $\sqrt{N}$, and $N$ is called a radical submodule if $\sqrt{N}=N$. The following result generalizes Theorem 1 in \cite{CimpriPrime}.  
	\begin{theorem}\label{thm:radical = semiprime}
		Let $M$ be a finitely generated module. For any proper submodule $N$  of $M$, the radical of $N$ is the smallest semiprime submodule of $M$ that contains $N$. In particular, $N$ is semiprime if and only if it is a radical submodule. 
	\end{theorem}
	\begin{proof}
		The first statement follows from the second statement because the intersection of any family of semiprime submodules is semiprime. To prove the second statement,  we may assume that $M= R^n/M'$ for some $n$ and a submodule $M'$ of $R^n$. Let $N/M'$ be a semiprime submodule of $M/M'$. By Proposition \ref{prop:correspondence}, $N$ is a semiprime submodule of $R^n$. By Theorem 1 in \cite{CimpriPrime}, $N$ is the intersection of a family of prime submodules $P_i$ of  $R^n$. Since each $P_i/M'$ is a prime submodule of $M/M'$ and $N/M'=\cap_i(P_i/M')$, the result follows.
	\end{proof}
	\begin{remark}
		Theorem \ref{thm:radical = semiprime} need not hold for not necessarily finitely generated modules over $R$. As an example, the reader can verify that the zero submodule of the $\mathbb{Z}$-module 
		\[
			M= \{\frac{m}{2^n}\,|\, m,n\in \mathbb{Z}\}
		\]
		is semiprime. However, the radical of the zero submodule equals the entire module since M contains no prime submodules. For more examples on modules with no prime submodules, see \cite{mccasland1997spectrum}.
	\end{remark}
	Next, we turn to the problem of describing the radical of a submodule. In \cite{McCasland1986}, McCasland and Moore introduced the notion of the ``envelope of a submodule" to give a description of the radical of a submodule. It turns out that this concept is effective only for a specific class of rings known as ``rings satisfying the radical formula" (see, for example, \cite{azizi2007radical,leung1997commutative, man1998commutative, McCaslandMoore}). In what follows, we use our notion of semiprimeness to give a more general description of the radical of a submodule of a finitely generated module. 
	
	Let $N$ be  a submodule of $M$. We define the \textit{first radical} of $N$, denoted by $N^{(1)}$, to be the submodule of $M$ generated by all $m\in M$ such that $m\in (N:m)M$. Inductively, we define the \textit{$i$-th radical} of $N$ using the formula 
	\[
		N^{(i)}:=\left(N^{(i-1)} \right)^{(1)},
	\]
	where $i\geq 2$. 
	It is clear that $N\subset N^{(1)}\subset\cdots\subset N^{(i)}\subset\cdots.$ We conclude the paper with the following result. 
	\begin{theorem}\label{thm:description of semiprime}
		For any submodule $N$ of a finitely generated module $M$, we have $\sqrt{N}=\cup_i N^{(i)}$. 
	\end{theorem}   
	 \begin{proof}
	 	We note that if $P$ is a prime submodule of $M$ containing $N$, then $N^{(1)}\subset P$ because $m\in N^{(1)}$ implies $m\in (N:m)M\subset (P:m)M$, which in turn implies $m\in P$ by Proposition \ref{prop:prime implies semiprime}. As a simple consequence of this observation, we see that $\cup_i N^{(i)}\subset \sqrt{N}$. To prove the reverse inclusion, we only need to show that $\cup_i N^{(i)}$ is semiprime, thanks to Theorem \ref{thm:radical = semiprime}. Let $m\in ( \cup_i N^{(i)}:M)M$. It follows that $m\in r_1M+\cdots+r_lM$ for some $r_1,\dots,r_l\in ( \cup_i N^{(i)}:M)$. Since $M$ is finitely generated, the condition $r_1,\dots,r_l\in ( \cup_i N^{(i)}:M)$ implies that there exists $N^{(n)}$ such that $r_1,\dots,r_l\in ( N^{(n)}:M)$. Therefore, we have 
	 	$m\in ( N^{(n)}:M)M$, which implies $m\in N^{(n+1)}$. It follows that $m\in \cup_i N^{(i)} $, completing the proof of the theorem. 
	 \end{proof}
\end{subsection}

%%%%%%%%%%%%%%%%%%%%%%%%%%%%%%%%%%%%%%%%%%%%%%%%%%%%%%%%%%%%%%%%%   
%%%%%%%%%%%%%%%%%%%%%%%%%%%%%%%%%%%%%%%%%%%%%%%%%%%%%%%%%%%%%%%%% 

\end{section}

\bibliographystyle{plain}
\bibliography{SPbiblan}

 \end{document}